\documentclass[10pt]{amsart}
\usepackage{amssymb}
\usepackage{amscd}
\usepackage{amsmath}
\theoremstyle{plain}
\newtheorem{theorem}{Theorem}

\theoremstyle{definition}
\newtheorem{question}{Question}

\newtheorem{proposition}{Proposition}[section]

\newtheorem{corollary}[proposition]{Corollary}
\newtheorem{lemma}[proposition]{Lemma}
\newtheorem{definition}[proposition]{Definition}
\theoremstyle{remark}
\newtheorem{claim}{Claim}

\DeclareMathOperator{\tcf}{tcf}

\DeclareMathOperator{\bd}{bd}

\newcommand{\pcfsig}{\text{\rm pcf}_{\sigma\text{\rm-com}}}
\newcommand{\ppgamma}{\text{\rm pp}_{\Gamma(\theta,\sigma)}}

\DeclareMathOperator{\Sk}{Sk}

\DeclareMathOperator{\cf}{cf}

\DeclareMathOperator{\PP}{PP}

\DeclareMathOperator{\pcf}{pcf}

\newcommand{\sk}{\vskip.05in}

\newcommand{\restr}{\upharpoonright}

\DeclareMathOperator{\cov}{cov}

\DeclareMathOperator{\reg}{{\sf Reg}}

\newcommand{\pcfgamma}{{\rm pcf}_{\Gamma(\theta,\sigma)}}

\DeclareMathOperator{\pp}{pp}

\numberwithin{equation}{section}

\begin{document}
\title{Representability and Compactness for Pseudopowers}
\author{Todd Eisworth}
\date{\today}
\begin{abstract}
We prove a compactness theorem for pseudopower operations of the form $\pp_{\Gamma(\mu,\sigma)}(\mu)$ where $\aleph_0<\sigma=\cf(\sigma)\leq\cf(\mu)$. Our main tool is a result that has Shelah's cov vs. pp Theorem as a consequence.  We also show that the failure of compactness in other situations has significant consequences for pcf theory, in particular, implying the existence of a progressive set $A$ of regular cardinals for which $\pcf(A)$ has an inaccessible accumulation point.
\end{abstract}
\keywords{pcf theory, psuedopowers, compactness}
\subjclass[2010]{03E04, 03E55}
\maketitle

\section{Background and Definitions}

This paper is concerned with problems arising in applications of pcf theory to cardinal arithmetic. Shelah's {\em Cardinal Arithmetic}~\cite{cardarith} is the most complete source for the background material we need, but we assume only that the reader has familiarity with the chapter of Abraham and Magidor~\cite{AM} in the Handbook of Set Theory~\cite{handbook}.  Any notation left undefined comes from their exposition.

To ground our discussion,  consider the following {\em ad hoc} definition.
Given a singular cardinal~$\mu$, let us agree to say a cardinal~$\kappa$ is {\em representable at $\mu$} if there are objects $A$ and $I$ such that
\begin{itemize}
\item $A$ is a cofinal subset of $\mu\cap\reg$ (the regular cardinals) with $|A|<\mu$
\sk
\item $I$ is a $\cf(\mu)$-complete ideal on $A$ extending $J^{\bd}[A]$, the ideal of bounded subsets of $A$, and
\sk
\item $\kappa=\tcf(\prod A/I)$ is the true cofinality of $\prod A/I$ , that is, there is an sequence $\langle f_\alpha:\alpha<\kappa\rangle$ of functions in $\prod A$ such that
\begin{itemize}
\item $\alpha<\beta<\kappa\Longrightarrow f_\alpha<_I f_\beta$, and
\sk
\item for all $g\in\prod A$ there is an $\alpha<\kappa$ such that $g<_I f_\alpha$.
\sk
\end{itemize}
\end{itemize}

Those readers familiar with Shelah's work in cardinal arithmetic will recognize this as related to the pseudopower operation~$\pp(\mu)$ and its variants, and consequently a lot is known already.  For example, the existence of scales tells us~$\mu^+$ is always representable at~$\mu$, as we can find $A$ of cardinality~$\cf(\mu)$ for which the corresponding ideal~$I$ is just the ideal of bounded subsets of~$A$.  A little more work shows that the set of cardinals representable at~$\mu$ is an interval of regular cardinals of length at most $2^{<\mu}\cdot\cf([\mu]^{<\mu},\subseteq)$ with some nice closure properties, as worked out by Shelah in Section~2 of~\cite{355}. Moreover, the supremum of the cardinals representable at~$\mu$ has a role in more standard cardinal arithmetic: if $\cf(\mu)$ is uncountable, then it is the minimum cardinality of a family $\mathcal{P}\subseteq[\mu]^{<\mu}$ such that every member of $[\mu]^{<\mu}$ is covered by a union of fewer than $\cf(\mu)$ sets from~$\mathcal{P}$\footnote{This follows easily from work of Shelah; see Corollary~\ref{6something} in this paper.}. The role  of this paper, though, is to address questions about ``compactness'', as exemplified by the following:

\begin{question}
Suppose $\kappa$ is a regular limit cardinal (that is, $\kappa$ is weakly inaccessible) and the set of $\tau$ that are representable at~$\mu$ is unbounded in $\kappa$.  Is~$\kappa$ also representable at~$\mu$?
\end{question}

We show the answer to the above question is ``yes'' in the case where $\mu$ has uncountable cofinality. We also examine more general versions of the above question, versions in which we impose restrictions on the type of representability under consideration. For example, we may restrict the cardinality of the set $A$, or relax the degree of completeness we demand from the ideal~$I$.  In many situations, we show the corresponding question has a positive answer unless something bizarre happens in pcf theory (in a sense to be made precise later).

Our main tool is a result that also implies Shelah's ``cov vs. pp. Theorem'', a theorem which connects pseudopowers with more traditional cardinal arithmetic considerations. His book {\em Cardinal Arithmetic} contains two proofs of this theorem, with the second proof claiming a positive answer to almost all versions of the compactness question we consider here. This second proof is found as Lemma~3.5 in Chapter~IX of {\em Cardinal Arithmetic}~\cite{cardarith}, however the proof given does not work (pointed out by Shelah in Section~6 of~\cite{513}).  What we do in this paper is to develop carefully the needed background material, and then push through a weaker version of his argument yielding the conclusions we mentioned above.   This has added importance because there are many results in his later work whose proofs state only ``like Lemma~3.5 of~\cite{400}''.\footnote{The original proof of the cov vs. pp Theorem (Theorem 5.4 of Chapter~III of~\cite{cardarith}) is perfectly fine, although it has a weaker conclusion and the methods do not seem to give the applications we derive here. That proof was also written before the existence of generators was proved, so it can also be simplified quite a bit.}

Returning to the required background for this paper, we can say a little more.  Our ideal reader will be familiar with pcf theory as put forward in~\cite{AM}, \cite{abc}, or the book~\cite{3germans}.  This may be a lot to demand, so we are deliberately gentle in a few places.  We do assume familiarity with the definition of $\pcf(A)$, the generators $B_\lambda[A]$, and the pcf ideals $J_{<\lambda}[A]$, as well as their basic properties.  We follow~\cite{AM} and~\cite{3germans} by using the adjective ``progressive'' to describe sets of cardinals $A$ for which $|A|<\min(A)$.  We follow Shelah and use the notation $J^{\bd}[A]$ to denote the ideal of bounded subsets of $A$.

We close this introduction with the following folklore proposition, important to us because it captures a basic fact used several times.  We leave the proof to the reader as a warm-up for the remainder of the paper.

\begin{proposition}
\label{folklore}
Let $A$ be a progressive set of regular cardinals, $\lambda$ a cardinal, and $J$ a proper ideal on $A$. Then the following are equivalent:
\begin{enumerate}
\item $\tcf(\prod A/J)$ is defined and equal to $\lambda$.
\sk
\item $J_{<\lambda}[A]\subseteq J$ and $A\setminus B_\lambda[A]\in J$	\end{enumerate}
\end{proposition}

Thus $\prod A/J$ has a true cofinality if and only if there is a $\lambda\in\pcf(A)$ with $A\setminus B_{\lambda}[A]\in J$, and the true cofinality is equal to the least such $\lambda$.  We will make use of this repeatedly.

\section{On $\pcf_{\Gamma(\theta,\sigma)}(A)$ and its relatives}

There are a couple of standard variants of $\pcf(A)$ discussed in~\cite{AM} that are special instances of the concepts we consider in this section, and we will use them to ground our discussion.  The first variant is $\pcf_\kappa(A)$ (see Definition 3.10 of~\cite{AM}) defined for $\kappa\leq |A|$ as
\begin{equation}
\pcf_\kappa(A)=\bigcup\{\pcf(X):X\in [A]^\kappa\}.	\end{equation}
Thus, a cardinal is in $\pcf_\kappa(A)$ if it is the cofinality of an ultraproduct of $A$ by an ultrafilter containing a set of size~$\kappa$.  The second variant, $\pcfsig(A)$, consists of those cardinals $\lambda$ such that
\begin{equation}
\lambda=\tcf(\prod A/ J)	
\end{equation}
where $J$ is a $\sigma$-complete ideal on $A$.  Abraham and Magidor give a nice discussion of these concepts, showing that they behave similar to $\pcf(A)$, and then relating them back to questions in combinatorial set theory.

We look at a common generalization of both of these, using the following definition and notation due to Shelah in Chapter~X of~\cite{cardarith}.  There are no surprises here, but we try to give systematic coverage because much of this has not been written down before.

\begin{definition}
Let $A$ be a progressive set of regular cardinals and let $\sigma<\theta$ be cardinals with $\sigma$ regular. A cardinal $\lambda$ is in $\pcf_{\Gamma(\theta,\sigma)}(A)$ if there is an ideal $J$ on $A$ such that
\begin{itemize}
\item $J$ is $\sigma$-complete,
\sk
\item $J^*$ (the filter dual to $J$) contains a set of cardinality $<\theta$, and
\sk
\item $\tcf(\prod A/ J)$ exists and is equal to $\lambda$.
\end{itemize}
Thus,
\begin{equation}
\label{thus}
\pcfgamma(A)=\bigcup\{\pcfsig(X):X\in [A]^{<\theta}\}.
\end{equation}
We will say $\lambda\in\pcf_{\Gamma^*(\theta,\sigma)}(A)$ if in addition $J$ extends the ideal of bounded subsets of~$A$.
\end{definition}

Whenever we use the notation $\Gamma(\theta,\sigma)$, it refers in some fashion to the family of $\sigma$-complete ideals on sets of cardinality less than~$\theta$.  This is false when taken in a literal sense: for example the ideals referenced in the preceding definition are technically on the set $A$, which can have large cardinality. However, the fact that we require the dual filter to contain a set of cardinality less than~$\theta$ means that our intuition is essentially correct.

The following elementary result records the fact that  $\pcf_{\Gamma(\theta,\sigma)}(A)$ reduces to standard variants of $\pcf(A)$ if we choose the parameters appropriately.

\begin{proposition}
With $A$, $\sigma$, and $\theta$ as above, we note:
\begin{enumerate}
\item $A\subseteq \pcf_{\Gamma(\theta,\sigma)}(A)\subseteq \pcf(A)$
\sk
\item If $\aleph_0\leq\sigma'\leq\sigma<\theta\leq\theta'$ with $\sigma'$ regular, then
\begin{equation}
\pcf_{\Gamma(\theta,\sigma)}(A)\subseteq  \pcf_{\Gamma(\theta',\sigma')}(A).
\end{equation}
\sk
\item For $\kappa\leq |A|$, $\pcf_\kappa(A)=\pcf_{\Gamma(\kappa^+,\aleph_0)}(A)$
\sk
\item If $|A|<\theta$, then
\begin{equation}
\pcf_{\Gamma(\theta,\sigma)}(A)=\pcfsig(A).
\end{equation}
In particular,
\begin{equation}
\pcf_{\Gamma(|A|^+,\aleph_0)}(A)=\pcf(A),
\end{equation}
and
\begin{equation}
\pcf_{\Gamma(|A|^+,\sigma)}(A)=\pcfsig(A).	
\end{equation}

\end{enumerate}
\end{proposition}

The specific goal of this section is to develop the theory of $\pcf_{\Gamma(\theta,\sigma)}(A)$ and its starred variant in a way that parallels the careful development of $\pcfsig(A)$ given in~\cite{AM}.  We first work towards characterizing when a cardinal $\lambda$ is in $\pcf_{\Gamma(\theta,\sigma)}(A)$.

\begin{definition}
Let $I$ be an ideal on a set $A$, and let $\theta$ be a cardinal.   We say $I$ is {\em $\theta$-based} if every $B\in I^+$ contains a subset of cardinality $<\theta$ that is also in $I^+$.  Equivalently, $I$ is $\theta$-based if
 \begin{equation}
 B\in I \Longleftrightarrow [B]^{<\theta}\subseteq I.
\end{equation}
\end{definition}

Given an ideal $I$ on a set $A$ and a cardinal $\theta$, we may define a collection $J\subseteq \mathcal{P}(A)$ by
\begin{equation}
B\in J\Longleftrightarrow [B]^{<\theta}\subseteq I.
\end{equation}
It is easy to show that $J$ is a $\theta$-based ideal extending $I$, and in fact $J$ is the smallest such ideal.  Moreover, if $I$ is $\sigma$-complete for some infinite regular cardinal $\sigma<\theta$, then the ideal $J$ is also $\sigma$-complete.

\begin{definition}
Let $I$ be an ideal on a set $A$, and let $\sigma<\theta$ be cardinals with $\sigma$ regular.  The {\em $(\theta,\sigma)$-completion of $I$} is the collection of sets $J\subseteq A$ consisting of those subsets $B$ of $A$ for which each element of $[B]^{<\theta}$ is covered by a union of fewer than $\sigma$ sets in $I$.
\end{definition}

The following simple proposition justifies our terminology, and has an elementary proof which is left to the reader.

\begin{proposition}
Let $I$ be an ideal on the set $A$, and let $\sigma<\theta$ be cardinals with $\sigma$~regular.  The $(\theta,\sigma)$-completion of $I$ is the smallest $\theta$-based $\sigma$-complete ideal on $A$ extending~$I$.
\end{proposition}

The next proposition is just an application of the idea of Proposition~\ref{folklore}, so we leave the easy proof to the reader. It pins down exactly when a cardinal makes it into $\pcfgamma(A)$.

\begin{proposition}
\label{pcfchar}
Suppose $A$ is a progressive set of regular cardinals, and $\sigma<\theta$ with $\sigma$ regular.  Then the following statements are equivalent for a cardinal~$\lambda$:
\begin{enumerate}
\item $\lambda\in\pcf_{\Gamma(\theta,\sigma)}(A)$.
\sk
\item $\lambda\in\pcf_{\sigma\text{-complete}}(B)$ for some $B\in [A]^{<\theta}$.
\sk
\item $\lambda\in\pcf(A)$ and the generator $B_\lambda[A]$ is not in the $(\theta,\sigma)$-completion of $J_{<\lambda}[A]$.
\sk
\item $\lambda=\tcf(\prod A/ J)$ for some $\theta$-based $\sigma$-complete ideal $J$ on $A$.
\end{enumerate}
\end{proposition}

In the same vein, we can characterize the $(\theta,\sigma)$-completion of $J_{<\lambda}[A]$ using pcf theory.

\begin{proposition}
Let $A$ be a progressive set of regular cardinals, let $\sigma<\theta$ with $\sigma$ regular, let $\lambda$ be regular, and let $J$ denote the $(\theta,\sigma)$-completition of $J_{<\lambda}[A]$. Then for $B\subseteq A$ we have
\begin{equation}
B\in J\Longleftrightarrow \pcfgamma(B)\subseteq\lambda.
\end{equation}
\end{proposition}
\begin{proof}
A subset $B$ of $A$ is in $J$ if and only if every member of $[B]^{<\theta}$ is in the $\sigma$-completion of $J_{<\lambda}[A]$.  By page 1211 of~\cite{AM}, this happens if and only if
\begin{equation}
C\in[B]^{<\theta}\Longrightarrow \pcf_{\text{$\sigma$-complete}}(C)\subseteq\lambda,
\end{equation}
and by (\ref{thus}) this is equivalent to
\begin{equation}
\pcfgamma(B)\subseteq\lambda.
\end{equation}
\end{proof}

Moving on, what can we say about $\pcf_{\Gamma^*(\theta,\sigma)}(A)$?  In the first place, this is not an interesting concept if $A$ happens to have a maximum element, as  $\pcf_{\Gamma^*(\theta,\sigma)}(A)$ will then consist of the single element $\max(A)$.  Secondly, if $A$ does not have a maximum element but either $\cf(\sup A)<\sigma$ or $\theta\leq\cf(\sup(A))$ holds, then $\pcf_{\Gamma^*(\theta,\sigma)}(A)$ will be empty: in the former situation, there are no proper $\sigma$-complete ideals on~$A$, and in the latter all sets of cardinality less than $\theta$ are bounded in~$A$.  Thus, we will usually assume
\begin{equation}
\sigma\leq\cf(\sup A)<\theta
\end{equation}
when we discuss $\pcf_{\Gamma^*(\theta,\sigma)}(A)$.  Another difference is that $A\subseteq\pcf_{\Gamma(\theta,\sigma)}(A)$ by way of the principal ideals, but this is no longer the case with $\Gamma^*(\theta,\sigma)$.  In fact, if $A$ does not have a maximum then $A$ and $\pcf_{\Gamma^*(\theta,\sigma)}(A)$ are disjoint.  Summarizing, we have:

\begin{proposition}
Suppose $A$ is a progressive set of regular cardinals, and $\sigma<\theta$ are cardinals with $\sigma$ regular.
\begin{enumerate}
\item $\pcf_{\Gamma^*(\theta,\sigma)}(A)\subseteq\pcfgamma(A)$.
\sk
\item $A\cap\pcf_{\Gamma^*(\theta,\sigma)}(A)=\emptyset$ unless $A$ has a maximum element, in which case $\pcf_{\Gamma^*(\theta,\sigma)}(A)=\{\max(A)\}$.
\sk
\item If $A$ does not have a maximum, then $\pcf_{\Gamma^*(\theta,\sigma)}(A)\neq\emptyset$ if and only if $\sigma\leq\cf(\sup A)<\theta$.
\sk
\end{enumerate}
\end{proposition}

Let us assume $A$ does not have a maximum and $\sigma\leq\cf(\sup A)<\theta$.  The theme of our discussion of $\pcfgamma(A)$ was that the $(\theta,\sigma)$-completion of $J_{<\lambda}[A]$ has the same role that~$J_{\lambda}[A]$ does in $\pcf(A)$.  For $\pcf_{\Gamma^*(\theta,\sigma)}(A)$, the analogous role is played by the $(\theta,\sigma)$-complete ideal generated by $J_{<\lambda}[A]$ together with the bounded subsets of~$A$. For now, let us call this ideal~$J$.  Note that membership in $J$ is easily described: a set $B$ is in $J$ if and only if there is a $\zeta<\sup(A)$ such that $B\setminus\zeta$ is in the $(\theta,\sigma)$-completion of $J_{<\lambda}[A]$.  Turning this around, we see $B\in J^+$ if and only if, for every $\zeta<\sup(A)$,  $B\setminus\zeta$ has a subset of cardinality less than~$\theta$ that is not covered by a union of fewer than $\sigma$ sets from $J_{<\lambda}[A]$. Since $\cf(\sup A)<\theta$, it follows that $B\in J^+$ if and only if $B$ has an unbounded set of cardinality less than~$\theta$ that cannot be covered by a union of fewer than~$\sigma$ sets from~$J_{<\lambda}[A]$.  This makes the following proposition easy, and we leave the proof to the reader.

\begin{proposition}
\label{aboveprop}
Let $A$ be a progressive set of regular cardinals without a maximum, and suppose $\sigma<\theta$ with $\sigma$ regular.  The following statements are equivalent for a cardinal $\lambda$ in $\pcf(A)$:
\begin{enumerate}
\item $\lambda\in\pcf_{\Gamma^*(\theta,\sigma)}(A)$
\sk
\item $B_\lambda[A]$ is not in the $(\theta,\sigma)$-complete ideal generated by $J_{<\lambda}[A]$ and $J^{\bd}[A]$
\sk
\item for all $\zeta<\sup(A)$, $B_\lambda[A]\setminus\zeta$ is not in the $(\theta,\sigma)$-completion of~$J_{<\lambda}[A]$.
\sk
\item $\lambda\in\pcfgamma(A\setminus\zeta)$ for all $\zeta<\sup(A)$.
\sk 
\item The generator $B_\lambda[A]$ has an unbounded subset of cardinality less than~$\theta$ that is not in the $\sigma$-completion of~$J_{<\lambda}[A]$.
\end{enumerate}
\end{proposition}

As an immediate corollary, we obtain the following useful result:

\begin{corollary}
\label{starchar}
Suppose $A$ is a progressive set of regular cardinals, and $\sigma<\theta$ with $\sigma$ regular.  Then
\begin{equation}
\pcf_{\Gamma^*(\theta,\sigma)}(A)=\bigcap_{\zeta<\sup A}\pcfgamma(A\setminus\zeta).
\end{equation}
\end{corollary}
\begin{proof}
If $A$ has a maximum, then both sides of the equation are equal to $\{\max(A)\}$.  If we are not in a vacuous situation, Proposition~\ref{aboveprop} gives us what we want.
\end{proof}

We close this section with a variant of a well-known result for $\pcf(A)$, and the proof uses a standard construction. The single appearance of $\Gamma$ rather than $\Gamma^*$ in the second part is not an error.

\begin{proposition}
\label{transitive}
Let $A$ be a progressive set of regular cardinals without a maximum, and let $\sigma<\theta$ with $\sigma$ regular.
\begin{enumerate}
\item If $B\subseteq\pcfgamma(A)$ and $|B|<\min(B)$, then $\pcf_{\Gamma(\cf(\theta),\sigma)}(B)\subseteq\pcf_{\Gamma(\theta,\sigma)}(A)$.
\sk
\item If $B\subseteq\pcf_{\Gamma^*(\theta,\sigma)}(A)$  and $|B|<\min(B)$, then $\pcf_{\Gamma(\cf(\theta),\sigma)}(B)\subseteq\pcf_{\Gamma^*(\theta,\sigma)}(A)$.
\end{enumerate}
\end{proposition}
\begin{proof}
For part (1), suppose $\lambda\in\pcf_{\Gamma(\cf(\theta),\sigma)}(B)$ witnessed by the ideal $I$.  For each $b\in B$ we choose an ideal $J_b$ witnessing~$b\in\pcfgamma(A)$, and let $J$ be the ideal on~$A$ defined by
\begin{equation}
X\in J\Longrightarrow\{b\in B:X\notin J_b\}\in I.
\end{equation}
Then
\begin{equation}
\lambda=\tcf(\prod A/ J)
\end{equation}
(see Claim~1.11 on page 12 of {\em Cardinal Arithmetic}, or the proof of Theorem~3.12 in~\cite{AM}), and it is readily checked that $J$ is $\sigma$-complete and the dual filter contains a set of size less than $\theta$ (this is the part that requires the use of $\Gamma(\cf(\theta),\sigma)$ on~$B$ rather than simply $\Gamma(\theta,\sigma)$).  Similarly, if each $J_b$ contains the bounded subsets of~$A$ then $J$ will as well, and (2) follows.
\end{proof}

If we limit ourselves to the case where $\theta$ is regular, we obtain the following corollary with a neater formulation.

\begin{corollary}
\label{regularcase}
Let $A$ be a progressive set of regular cardinals without a maximum, and let $\sigma<\theta$ be regular cardinals.
\begin{enumerate}
\item If $B\subseteq\pcfgamma(A)$ and $|B|<\min(B)$, then $\pcfgamma(B)\subseteq\pcfgamma(A)$.
\sk
\item If $B\subseteq\pcf_{\Gamma^*(\theta,\sigma)}(A)$  and $|B|<\min(B)$, then $\pcfgamma(B)\subseteq\pcf_{\Gamma^*(\theta,\sigma)}(A)$.
\end{enumerate}
\end{corollary}

\section{On $\sup\pcfgamma(A)$ and $\sup\pcf_{\Gamma^*(\theta,\sigma)}(A)$}

Our attention now turns to characterizing the cardinal $\sup\pcf_{\Gamma(\theta,\sigma)}(A)$ (and its starred variant) for a progressive set of regular cardinals~$A$. Part of our work repeats some material from Section~3 of~\cite{400} (Chapter~X of~\cite{cardarith}). In particular,  Proposition~\ref{supchar} below is essentially Claim~3.2 from~\cite{400} but the proof  we gives runs smoother, as we can take advantage of the work done in the previous section.

\begin{definition}
Let $A$ be a non-empty set of ordinals.  A set $F\subseteq\prod A$ is said to be a
$(\theta,\sigma)$-cover of $\prod A$ if for any $B\subseteq A$ of cardinality $<\theta$ and $g\in\prod B$, there is an $F_0\subseteq F$ such that
\begin{equation}
|F_0|<\sigma
\end{equation}
and
\begin{equation}
(\forall b\in B)(\exists f\in F_0)[g(b)<f(b)].
\end{equation}
The $(\theta,\sigma)$-cofinality of $\prod A$, denoted $\cf^\sigma_{<\theta}(\prod A)$, is defined to be the minimum size of a $(\theta,\sigma)$-cover of~$\prod A$.  These
notions extend in a natural way to structures of the form $\prod A/ J$ for an ideal~$J$ as well.
\end{definition}

Cardinals of the form $\cf^\sigma_{<\theta}(\prod A)$ for various sets $A$ are a major player in Chapter~X of {\em Cardinal Arithmetic}, and much more information can be found there.  

\begin{proposition}[Claim 3.2 of~\cite{400}]
\label{supchar}
Let $A$ be a progressive set of regular cardinals.  Then
\begin{equation}
\sup\pcf_{\Gamma(\theta,\sigma)}(A)=\cf^\sigma_{<\theta}(\prod A).
\end{equation}
\end{proposition}
\begin{proof}
One direction (``$\leq$'') of this proposition is very easy, so we concentrate on the harder direction.

For each $\lambda\in\pcf(A)$, we fix a sequence $\langle f^\lambda_\alpha:\alpha<\lambda\rangle$ in $\prod A$ such that for any $g\in\prod A$ there is an $\alpha<\lambda$ with $g\restr B_\lambda[A]<f^\lambda_\alpha\restr B_\lambda[A]$.\footnote{This is standard pcf theory: see Theorem 4.4 of~\cite{AM}, and use the fact that $\lambda=\max\pcf(B_\lambda[A])$.} Next,  we define
\begin{equation}
F:=\bigcup_{\lambda\in\pcf_{\Gamma(\theta,\sigma)}(A)}\{f^\lambda_\alpha:\alpha<\lambda\}.
\end{equation}
Clearly $|F|\leq\sup\pcf_{\Gamma(\theta,\sigma)}(A)$, so we finish if we show $F$ is $(\theta,\sigma)$-cofinal in $\prod A$.

Suppose this were not the case, as witnessed by the function $g\in\prod A$. We define a collection $X$ of subsets of $A$ by setting
\begin{equation}
B\in X \Longleftrightarrow (\exists f\in F)(\forall b\in B)[g(b)<f(b)],
\end{equation}
and let $I$ be the $(\theta,\sigma)$-complete ideal on $A$ generated by $X$.

Our assumption on $g$ implies that $I$ is in fact a proper $(\theta,\sigma)$-complete ideal on~$A$, so we can choose $\lambda\in\pcf(A)$ least such that $B_\lambda[A]\notin I$.
Note that this choice of~$\lambda$ implies
\begin{equation}
J_{<\lambda}[A]\subseteq I,
\end{equation}
and so $B_\lambda[A]$ is not in the $(\theta,\sigma)$-completion of $J_{<\lambda}[A]$.  We conclude
\begin{equation}
\lambda\in\pcf_{\Gamma(\theta,\sigma)}(A)
\end{equation}
by way of Proposition~\ref{pcfchar}, and so there is an $f\in F$ (in fact, $f= f^\lambda_\alpha$ for some $\alpha<\lambda$) such that $g(b)<f(b)$ for all $b\in B_\lambda[A]$.
But then $B_\lambda[A]$ would be in $X\subseteq I$, and this contradicts our choice of~$\lambda$.
\end{proof}

We take note of the following special case as being of independent interest, used several times in the literature (see for example Theorem~8.6 of~\cite{AM}, or some of the proofs in~\cite{460}).
\begin{corollary}
Let $A$ be a progressive set of regular cardinals, and let $\sigma$ be regular.  Then $\sup\pcfsig(A)$ is the $<\sigma$-cofinality of $\prod A$, that is, the minimum cardinality of a family $F\subseteq\prod A$ such that the functions formed by taking the supremum of fewer than~$\sigma$ functions from $F$ are cofinal in~$\prod A$.
\end{corollary}

We need to extend the above work in order to characterize $\sup\pcf_{\Gamma^*(\theta,\sigma)}(A)$ in a similar way.  For the remainder of this section, we assume $\sigma<\theta$ are both regular cardinals, and we will not bother to track what occurs if $\theta$ is singular.

\begin{proposition}
\label{tangerine}
If $A$ is a progressive set of regular cardinals and $\pcf_{\Gamma^*(\theta,\sigma)}(A)$ is non-empty, then
\begin{equation}
\label{lastminute}
\sup\pcf_{\Gamma^*(\theta,\sigma)}(A)=\lim_{\zeta<\sup A}\sup\pcfgamma(A\setminus\zeta).
\end{equation}
\end{proposition}
\begin{proof}
If $A$ has a maximum, then both sides are equal to~$\max(A)$.  Thus, we may assume $A$ is unbounded in $\sup(A)$ and $\sigma\leq\cf(\sup(A))<\theta$.  One direction (``$\leq$'') of our desired inequality is immediate by Corollary~\ref{starchar}, so we work on the other direction.

Notice that the sequence $\langle\sup\pcfgamma(A\setminus\zeta):\zeta<\sup A\rangle$ is non-increasing, so it is eventually constant.  Removing an initial segment from $A$ does not affect $\pcf_{\Gamma^*(\theta,\sigma)}(A)$, so we may as well assume the sequence is constant with value some~$\kappa$, that is,
\begin{equation}
\sup\pcfgamma(A\setminus\zeta)=\kappa\text{ for all }\zeta<\sup A.
\end{equation}
Given $\tau<\kappa$, it suffices to produce some $\lambda\in\pcf_{\Gamma^*(\theta,\sigma)}(A)$ with $\tau\leq\lambda$.

Let us fix an increasing sequence $\langle\zeta_\alpha:\alpha<\cf(\sup A)\rangle$ cofinal in~$A$, and define
\begin{equation}
\lambda_\alpha:=\min(\pcfgamma(A\setminus\zeta_\alpha))\setminus\tau.
\end{equation}
Our assumptions guarantee that $\lambda_\alpha$ is defined, and the resulting sequence of cardinals is non-decreasing.

By thinning things out, we may assume that the sequence is either constant, or strictly increasing.  In the former case, we are done as the corresponding cardinal is in~$\pcf_{\Gamma^*(\theta,\sigma)}(A)$ by Corollary~\ref{starchar}.  Thus, we can assume that the various $\lambda_\alpha$ are distinct.

Let $B=\{\lambda_\alpha:\alpha<\cf(\sup A)\}$.  Since $\cf(\sup B)=\cf(\sup A)$, we know that $\pcf_{\Gamma^*(\theta,\sigma)}(B)$ is non-empty.  Given $\lambda$ in this set,
Corollary~\ref{regularcase} tells us
\begin{equation}
\lambda\in\pcf_{\Gamma^*(\theta,\sigma)}(A)
\end{equation}
as well.  Since $\tau$ must be less than $\lambda$, we are done.
\end{proof}

\begin{corollary}
\label{yay}
If $A$ is a progressive set of regular cardinals and $\sup\pcf_{\Gamma^*(\theta,\sigma)}(A)$ is non-empty, then
\begin{equation}
\sup\pcf_{\Gamma^*(\theta,\sigma)}(A)=\sup\pcf_{\Gamma(\theta,\sigma)}(A\setminus\zeta),
\end{equation}
for all sufficiently large $\zeta<\sup(A)$.   In particular, there is a $\zeta<\sup A$ so that
\begin{equation}
\sup\pcf_{\Gamma^*(\theta,\sigma)}(A)=\cf^\sigma_{<\theta}(\prod A\setminus\zeta).
\end{equation}
\end{corollary}

Finally, we come to the promised characterization of $\sup\pcf_{\Gamma^*(\theta,\sigma)}(A)$:

\begin{theorem}
\label{starcof}
Let $A$ be a progressive set of regular cardinals with $\pcf_{\Gamma^*(\theta,\sigma)}(A)$  non-empty. Then
\begin{equation}
\sup\pcf_{\Gamma^*(\theta,\sigma)}(A)=\cf^\sigma_{<\theta}(\prod A/ J^{\bd}[A]).
\end{equation}
\end{theorem}
\begin{proof}
Given Corollary~\ref{yay}, it suffices to prove there is a  $\zeta<\sup(A)$ such that
\begin{equation}
\cf^\sigma_{<\theta}(\prod A/ J^{\bd}[A]) =\cf^\sigma_{<\theta}\bigl(\prod (A\setminus\zeta)\bigr).
\end{equation}
Any family $F$ that is $(\theta,\sigma)$-cofinal in a tail of $A$ will be $(\theta,\sigma)$-cofinal in $\prod A/J^{\bd}[A]$, so we need only show
that if $F\subseteq\prod A$ is $(\theta,\sigma)$-cofinal in $\prod A/J^{\bd}[A]$, then there is a $\zeta<\sup(A)$ such that $F$ is  $(\theta,\sigma)$-cofinal in $\prod A\setminus\zeta$.  Suppose this were not the case, and let $\langle\zeta_\alpha:\alpha<\cf(\sup(A))\rangle$ be increasing and cofinal in~$\sup(A)$.  For each $\alpha<\cf(\sup(A))$,  we can choose a function $g_\alpha\in\prod A$ such that for any subset $F_0$ of $F$ of cardinality less than~$\sigma$, there is an $a\in A\setminus\zeta_\alpha$ such that
\begin{equation}
\label{eqncon}
(\forall f\in F_0)\left[f(a)\leq g_\alpha(a)\right].
\end{equation}
Since $\cf(\sup(A))\leq|A|<\min(A)$, it follows that
\begin{equation}
g:=\sup\{g_\alpha:\alpha<\cf(\sup(A))\}
\end{equation}
is in $\prod A$.  By our assumptions, there is a family $F_0\subseteq F$ of cardinality $<\sigma$ and $\zeta<\sup(A)$ such that
\begin{equation}
\label{eqncon2}
(\forall a\in A\setminus\zeta)(\exists f\in F_0)[g(a)<f(a)].
\end{equation}
This quickly yields a contradiction:  if we choose $\alpha$ with $\zeta<\zeta_\alpha$, then there is an $a\in A\setminus\zeta$ such that~(\ref{eqncon}) holds for our choice of~$F_0$, and since $g_\alpha(a)\leq g(a)$ we contradict~(\ref{eqncon2}).
\end{proof}

\section{On $\pp_{\Gamma(\theta,\sigma)}(\mu)$}

In this section, we move beyond considering the structure of $\pcf_{\Gamma(\theta,\sigma)}(A)$ for progressive $A$, and look instead at more general statements in cardinal arithmetic involving pseudopowers $\pp_{\Gamma(\theta,\sigma)}(\mu)$ for singular cardinals~$\mu$.   We have touched on these matters in the introduction, but now we give the official definiton:
\begin{definition}
Suppose $\mu$, $\theta$, and $\sigma$ are infinite cardinals with $$\sigma=\cf(\sigma)\leq\cf(\mu)<\theta\leq\mu.$$ $\PP_{\Gamma(\theta,\sigma)}(\mu)$ is the collection of all cardinals $\lambda$ such that there are a  cofinal $A\subseteq\mu\cap\reg$ of cardinality $<\theta$ and a $\sigma$-complete
ideal $J$ on $A$ containing the bounded subsets of $A$ with
\begin{equation}
\lambda=\tcf(\prod A/ J).
\end{equation}
We define
\begin{equation}
\pp_{\Gamma(\theta,\sigma)}(\mu)=\sup\PP_{\Gamma(\theta,\sigma)}(\mu).
\end{equation}
\end{definition}
The restriction to  $\sigma\leq \cf(\mu)<\theta$ is in place so that we avoid trivialities.
Also, these notions are clearly connected to the things we discussed in the preceding section.  For example, if $\mu$ is singular, $A$ is cofinal in $\mu\cap\reg$, and $\sigma\leq\cf(\mu)<\theta$, then
\begin{equation}
\label{hah}
\pcf_{\Gamma^*(\theta,\sigma)}(A)\subseteq\PP_{\Gamma(\theta,\sigma)}(\mu).
\end{equation}
And conversely, $\lambda\in \PP_{\Gamma(\theta,\sigma)}(\mu)$ by definition means there is an $A$ cofinal in $\mu\cap\reg$ such that
\begin{equation}
\lambda\in\pcf_{\Gamma^*(\theta,\sigma)}(A).
\end{equation}

Thus, $\ppgamma(\mu)$ is the supremum of all cardinals that appear in $\pcf_{\Gamma^*(\theta,\sigma)}(A)$ for some cofinal $A\subseteq\mu\cap\reg$ (and we can even require $|A|<\theta$).   Moreover, for any $A$ cofinal in $\mu\cap\reg$ satisfying $|A|<\mu$, we know
\begin{equation}
\cf^\sigma_{<\theta}\left(\prod A/J^{\bd}[A]\right)\leq\ppgamma(\mu)
\end{equation}
by way of Theorem~\ref{starcof}.  In Theorem~\ref{theorem1} below, we show that the ideal $J^{\bd}[A]$ can be removed if $\min(A)$ is sufficiently large.

Shelah works out many basic properties of~$\PP_{\Gamma(\theta,\sigma)}(\mu)$ and $\pp_{\Gamma(\theta,\sigma)}(\mu)$ in Section~3 of Chapter~3 in~\cite{cardarith}. In particular, he establishes $\PP_{\Gamma(\theta,\sigma)}(\mu)$ is an interval of regular cardinals with minimum~$\mu^+$.   Our aim in this section is to make sure some useful properties are documented; most of what we do here is implicit in Shelah's work even if he never explicitly states the results.  Again, the groundwork we did in the previous section makes most of the proofs here quite easy.

\begin{theorem}
\label{theorem1}
Suppose $\mu$ is singular, and $\aleph_0\leq\sigma\leq\cf(\mu)<\cf(\theta)\leq\theta<\mu$.  Then there is an $\eta<\mu$
such that
\begin{equation}
\sup\pcfgamma(A)\leq\ppgamma(\mu)
\end{equation}
whenever $A$ is a progressive subset of $[\eta,\mu)\cap\reg$.
\end{theorem}

It should be clear that whenever $A$ is a cofinal progressive subset of $\mu\cap\reg$, then
\begin{equation}
\sup\pcf_{\Gamma^*(\theta,\sigma)}(A)\leq\ppgamma(\mu)	
\end{equation}
just by the definitions involved, but the conclusion of the theorem is more general.  It speaks about all progressive subsets of $[\eta,\mu)\cap\reg$, not just the unbounded ones, and moves us from $\Gamma^*(\theta,\sigma)$ to $\Gamma(\theta,\sigma)$. Said another way, the theorem asserts that $\pp_{\Gamma(\theta,\sigma)}(\mu)$ provides a bound on the $(\theta,\sigma)$-cofinality of any progressive $A$ drawn from the tail $[\eta,\mu)\cap\reg$.

\begin{proof}
Let $\kappa=\ppgamma(\mu)$, and let $\langle\mu_\alpha:\alpha<\cf(\mu)\rangle$ be increasing and cofinal in ~$\mu$.  Assume by way of contradiction that the theorem fails,
so for each $\alpha<\cf(\mu)$ we can find $A_\alpha\subseteq (\mu_\alpha,\mu)\cap\reg$  such that
\begin{equation}
|A_\alpha|<\theta,
\end{equation}
and
\begin{equation}
\kappa<\sup\pcfgamma(A_\alpha).
\end{equation}
If we define
\begin{equation}
A:=\bigcup_{\alpha<\cf(\mu)} A_\alpha,
\end{equation}
then $|A|<\theta$, $A$ is cofinal in $\mu\cap\reg$, and for all $\zeta<\mu=\sup(A)$, we have
\begin{equation}
\kappa<\sup\pcfgamma(A\setminus\zeta).
\end{equation}
But this implies
\begin{equation}
\kappa<\sup\pcf_{\Gamma^*(\theta,\sigma)}(A)
\end{equation}
by Proposition~\ref{tangerine}.  But we have a contradiction: (\ref{hah}) tells us
\begin{equation}
\pcf_{\Gamma^*(\theta,\sigma)}(A)\subseteq\PP_{\Gamma(\theta,\sigma)}(A)
\end{equation}
and so
\begin{equation}
\sup\pcf_{\Gamma^*(\theta,\sigma)}(A)\leq\ppgamma(A)=\kappa,
\end{equation}
as well, but this is absurd.
\end{proof}

\noindent Note that in fact this theorem yields a characterization of~$\ppgamma(\mu)$ as a type of $\lim\sup$.

\begin{corollary}
With $\mu$, $\theta$, and $\sigma$  as above, the following two cardinals are equal:
\begin{enumerate}
\item $\ppgamma(\mu)$, and
\sk
\item $\min\{\sup\{\cf^\sigma_{<\theta}(\prod A):A\subseteq(\eta,\mu)A\cap{\sf Reg}, |A|<\min(A)\}:\eta<\mu\}$.
\end{enumerate}
\end{corollary}
\noindent Thinking about this in another way, if we define
\begin{equation}
\label{kappaeta}
\kappa_\eta:=\sup\{\sup\pcfgamma(A):A\subseteq(\eta,\mu)\cap{\sf Reg}, |A|<\min(A)\}
\end{equation}
for $\eta<\mu$, then the sequence $\langle \kappa_\eta:\eta<\mu\rangle$ is non-increasing and hence eventually constant.  This limiting value is just~$\ppgamma(\mu)$.  (There is nothing special about our use of $\Gamma(\theta,\sigma)$ here, as similar things can be proved for $\pp(\mu)$, or other variants. All these ideas are certainly present in Shelah's work; we are just giving them a crisp formulation.)  Stronger results hold if $\sigma$ is uncountable, see Corollary~\ref{triple}.

It also makes sense to ask about the value of $\kappa_\eta$ from~(\ref{kappaeta}) for various $\eta<\mu$,  and these cardinals admit an easy description as well:

\begin{theorem}
Suppose $\sigma\leq\cf(\mu)<\theta\leq\eta<\mu$ with $\sigma$ regular.  Then $\lambda_1=\lambda_2$, where
\begin{equation}
\lambda_1:=\sup\{\cf^\sigma_{<\theta}(\prod A):A\subseteq[\eta,\mu)\cap{\sf Reg}, |A|<\min(A)\},
\end{equation}
and
\begin{equation}
\lambda_2:=\sup\{\ppgamma(\tau):\eta\leq\tau\leq\mu\text{ and }\sigma\leq\cf(\tau)<\theta\}.
\end{equation}
\end{theorem}
\begin{proof}
It suffices to prove for a cardinal~$\lambda$ that $\lambda\in\pcfgamma(A)$ for some progressive $A\subseteq [\eta,\mu)$ if and only if $\lambda\in\PP_{\Gamma(\theta,\sigma)}(\tau)$ for some singular cardinal $\tau\leq\mu$ with $\sigma\leq\cf(\tau)<\theta$.

This follows quite easily by a standard argument: given $\lambda\in\pcf\Gamma(A)$ with $A$ as above, we can assume without loss of generality that $|A|<\theta$, and that there is a $\sigma$-complete ideal~$J$ on $A$ such that $J$ contains all initial segments of $A$, and
\begin{equation}
\lambda=\tcf(\prod A/J).
\end{equation}
We then let $\tau=\sup (A)$ and note that $\tau$ has the needed properties and $\lambda$ is in $\PP_{\Gamma(\theta,\sigma)}(\tau)$.  The other direction is even easier.
\end{proof}

In particular, for the situation where $\eta=\theta$, we see that given a progressive $A\subseteq [\theta,\mu)\cap\reg$ there is a singular cardinal $\tau$ such that
\begin{enumerate}
\item $\sigma\leq\cf(\tau)<\theta<\tau\leq\mu$, and
\sk
\item $\sup\pcfgamma(A)\leq\ppgamma(\tau)$.
\end{enumerate}
Thus, $\sup\{\ppgamma(\tau):\sigma\leq\cf(\mu)<\theta\leq\tau\leq\mu\}$ bounds the $(\theta,\sigma)$-cofinality of any progressive subset $A$ of $[\theta,\mu)\cap\reg$.  This is certainly not a new result as Shelah routinely uses this fact in his work in this area, but it is very useful. Since it does not appear to be stated explicitly in the literature, we do so here for future convenience and ease of reference.

\begin{corollary}
Suppose $\sigma\leq\cf(\mu)<\theta<\mu$, with $\sigma$ regular.
Then for any progressive $A\subseteq [\theta,\mu)\cap \reg$ we have
\begin{equation}
\cf^\sigma_{<\theta}(\prod A)\leq \sup\{\pp_{\Gamma(\theta,\sigma)}(\tau):\sigma\leq\cf\tau<\theta\leq\tau\leq\mu\}.
\end{equation}
\end{corollary}

\section{Main Theorem}

In this section, we prove our version of Theorem~3.5 from Chapter~X of~\cite{cardarith}, providing a partial repair of the incorrect proof given there. We take this opportunity to remind the reader of the definition of the covering numbers $\cov(\mu,\kappa,\theta,\sigma)$, which play a prominent role in~{\em Cardinal Arithmetic} and its continuations.

\begin{definition}
Let $\mu\geq\kappa\geq\theta>\sigma\geq 2$ be cardinals.
\begin{enumerate}
\item  A subset $\mathcal{P}$ of $[\mu]^{<\mu}$ is said to $\sigma$-cover $[\mu]^{<\theta}$ if any member of $[\mu]^{<\theta}$ is covered by some union of fewer than $\sigma$ sets from $\mathcal{P}$.
\sk
\item $\cov(\mu,\kappa,\theta,\sigma)$ is defined to be the minimum cardinality of a family $$\mathcal{P}\subseteq [\mu]^{<\kappa}\subseteq [\mu]^{<\mu}$$ that $\sigma$-covers $[\mu]^{<\theta}$.
\end{enumerate}
\end{definition}

The assumptions on the relative sizes of the parameters in the above definition are made to avoid trivialities.  Many of the basic properties of these covering numbers are laid out in Section 5 of Chapter II of~\cite{cardarith}; for our purposes, it is enough to note that they represent a natural way of measuring the size of $[\mu]^{<\mu}$ and its relatives.  The reader may also refer to the so-called Analytical Guide\footnote{This appears in published form at the end of the book~\cite{cardarith}, but the document -- known as [E:12] -- has been updated many times, and is available on Shelah's Archive.} to {\em Cardinal Arithmetic} for a broader discussion of the importance of these cardinals.

\begin{theorem}
\label{modelcover}
Assume $\sigma$, $\theta$, and $\mu$ are cardinals with $\sigma$ and $\theta$ regular such that
\begin{equation}
\aleph_0<\sigma\leq\cf(\mu)<\theta<\mu,
\end{equation}
and let $M$ be an elementary submodel of $H(\chi)$ for some sufficiently large regular $\chi$ with $\mu+1\subseteq M$.  Further assume

\begin{list}{$\circledast$}{\setlength{\leftmargin}{.5in}\setlength{\rightmargin}{.5in}}
\item If $A\in M$ is a subset of $\mu\cap\reg$ with $|A|<\mu$, then $M\cap\prod A$ is $(\theta,\sigma)$-cofinal in $\prod A$ modulo the ideal of bounded subsets of~$\mu$.
\end{list}
Then $M\cap[\mu]^{<\mu}$ is a $\sigma$-cover of $[\mu]^{<\theta}$.
\end{theorem}

Notice that we assume $\sigma$ is uncountable, and this will be crucial in our proof.  Clearly this theorem gives us information on covering numbers, as
\begin{equation}
\cov(\mu,\mu,\theta,\sigma)\leq |M|
\end{equation}
in the above situation.

\begin{proof}
With $M$ be as in the statement of the theorem, let us define
\begin{equation}
\mathcal{F}=M\cap\prod(\mu\cap\reg),
\end{equation}
and
\begin{equation}
\mathcal{P}=M\cap [\mu]^{<\mu}.
\end{equation}
Given an unbounded subset $A$ of $\mu\cap\reg$, we will abuse notation a little and say that $\mathcal{F}$ is $(\theta,\sigma)$-cofinal in $\prod A$ when what we really mean
is that the family $\{f\restr A:f\in\mathcal{F}\}$ is $(\theta,\sigma)$-cofinal in $\prod A$.  Our assumption $\circledast$ therefore says that $\mathcal{F}$ is $(\theta,\sigma)$-cofinal in $\prod A$ modulo the ideal of sets bounded in~$\mu$ whenever $A\in M$ is a subset of $\mu\cap\reg$ of cardinality less than~$\theta$.  Our first step is the following easy observation, which leverages our assumption $\circledast$ to obtain a stronger property.

\begin{claim}
\label{improved}
Suppose $A$ is a subset of $\mu\cap\reg$ that can be covered by a union of fewer than $\sigma$ sets from $\mathcal{P}$.
Then $\mathcal{F}$ is $(\theta,\sigma)$-cofinal in $\prod A/J^{\bd}_\mu$.
\end{claim}
\begin{proof}
Suppose $\sigma^*<\sigma$, $A\subseteq \bigcup_{\alpha<\sigma^*}A_\alpha$ where $A_\alpha\in\mathcal{P}$ for each $\alpha$, and $g\in\prod A$.
For each $\alpha$ we choose $F_\alpha\subseteq\mathcal{F}$ and $\mu_\alpha<\mu$ such that
\begin{equation}
(\forall a\in A_\alpha\cap A\setminus\mu_\alpha)(\exists f\in F_\alpha)[g(a)<f(a)],
\end{equation}
and then let
\begin{equation}
F:=\bigcup_{\alpha<\sigma^*}F_\alpha
\end{equation}
and
\begin{equation}
\mu^*:=\sup\{\mu_\alpha:\alpha<\sigma^*\}.
\end{equation}
Since $\sigma^*<\sigma=\cf(\sigma)\leq\cf(\mu)$ we know $|F|<\sigma$, $\mu^*<\mu$, and
\begin{equation}
(\forall a\in A\setminus\mu^*+1)(\exists f\in F)[g(a)<f(a)],
\end{equation}
as required.
\end{proof}

We turn now to the proof of the theorem.  Let $X$ be a subset of $\mu$ of cardinality less than $\theta$.  We are going to use induction on $k<\omega$ to define objects $H_k$, $N^a_k$, $N^b_k$, and $A_k$, $B_k$, where
\begin{itemize}
\item $H_k$ is a certain subset of $\mathcal{F}$ of cardinality less than $\sigma$,
\sk
\item $\mu_k$ is a cardinal less than $\mu$,
\sk
\item $N_k^a = \Sk(\mu_k+1\cup\{\mu\}\cup H_k)$,
\sk
\item $N^b_k = \Sk(X\cup\{\mu\}\cup H_k)$,
\sk
\item $A_k = N^a_k\cap [\mu_k^+,\mu)\cap{\sf Reg}$,
\sk
\item $B_k = N^b_k \cap A_k$
\sk
\item $g_k$ is the characteristic function of $N_k^b$ in $\prod B_k$, that is, for $\iota\in B_k$ we have
 \begin{equation}
 g_k(\iota):=\sup(N_k^b\cap\iota),
 \end{equation}
\sk
\item $\mu_k<\mu_{k+1}$ and $(\forall \iota\in B_k\setminus \mu_{k+1})(\exists f\in H_{k+1})[g_k(\iota)<f(\iota)].$
\end{itemize}

How do we do this? If we start with $H_0=\emptyset$ and $\mu_0=0$, then we have all the other objects for $k=0$.
Suppose now that we have defined things through stage~$k$.  Note that $A_k$ is a progressive subset of $\mu\cap\reg$ of cardinality $\mu_k$, and $B_k$ is a subset of $A_k$ with $|B_k|<\theta$.
The following claim is the key step which allows our construction to continue:

\begin{claim}
\label{claim}
$A_k$ is covered by a union of fewer than~$\sigma$ sets from~$\mathcal{P}$.  Thus,
$\mathcal{F}$ is $(\theta,\sigma)$-cofinal in $\prod A_k/J^{\bd}_\mu$.
\end{claim}
\begin{proof}
The proof is a fairly standard  Skolem hull argument, sketched below

Given $x\in N^a_k$, note that there is a set $X$ such that
\begin{itemize}
\item $x\in X$,
\sk
\item $|X|\leq\eta$, and
\sk
\item $X\in\Sk(\{\eta,\mu\}\cup H_k)$.
\sk
\end{itemize}
Why? Given the definition of $N^a_k$, there are a formula $\varphi$, ordinals $\eta_1,\dots,\eta_i$ from $\mu_k$, and parameters $p_1$, ..., $p_k$ from $\{\mu_k,\mu\}\cup H_k$
such that $x$ is the unique $y$ for which
\begin{equation}
\label{eqn}
H(\chi)\models\varphi(y,\vec{\eta},\vec{p}).
\end{equation}
Define $X$ to be the collection of all $z$ for which there are $\eta_1,\dots,\eta_i<\mu_k$ such that $z$ is the unique $y$ such that~(\ref{eqn}) holds.  Clearly $|X|\leq\mu_k$, and $X$ is definable in $H(\chi)$
by a formula with parameters from $\{\mu_k,\mu\}\cup H_k$.  Since the cardinality of $\Sk(\{\mu_k,\mu\}\cup H_k)$ is less than~$\sigma$ and
\begin{equation}
\Sk(\{\mu_k,\mu\}\cup H_k)\subseteq M,
\end{equation}
we can conclude that $A_k$ can be covered by a union of fewer than~$\sigma$ sets from $M\cap [\mu]^{\leq\mu_k}\subseteq\mathcal{P}$.
\end{proof}

Given the above claim, we know that there is an $F\subseteq \mathcal{F}$ and cardinal $\mu_{k+1}<\mu$ such that
\begin{itemize}
\item $|F|<\sigma$,
\sk
\item $\mu_k<\mu_{k+1}$, and
\sk
\item $(\forall \iota\in B_k)(\exists f\in F)[g_k(\iota)<f(\iota)]$.
\sk
\end{itemize}
Now the construction continues once we set $H_{k+1}=H_k\cup F$.

To finish the proof of the theorem, let

\begin{equation}
\mu^*=\sup\{\mu_k:k<\omega\},
\end{equation}

\begin{equation}
N^a:=\bigcup_{k<\omega} N_k^a,
\end{equation}
and
\begin{equation}
N^b:=\bigcup_{k<\omega} N^b_k.
\end{equation}

Notice that since $\cf(\sigma)>\aleph_0$ and
\begin{equation}
N^a\cap\mu=\bigcup_{k<\omega}N^a_k\cap\mu,
\end{equation}
we know $N^a\cap\mu$ is covered by a union of fewer than $\sigma$ sets from $\mathcal{P}$ by way Claim~\ref{claim}. Since $X\subseteq N^b\cap\mu$ by construction, the following
proposition proved using a standard argument of Shelah will let us finish:

\begin{proposition}
\label{goalprop}
$N^b\cap\mu\subseteq N^a\cap\mu$.
\end{proposition}
\begin{proof}
Suppose this fails, and define
\begin{equation}
\gamma(*)=\min(N^b\cap\mu\setminus N^a).
\end{equation}
Since $\cf(\mu)=\mu_0\leq \mu^*$ we know $N^a\cap \mu$ is unbounded in $\mu$.  In particular,
\begin{equation}
N^a\cap\mu\setminus \gamma(*)\neq\emptyset.
\end{equation}
Thus, we may define
\begin{equation}
\beta(*)=\min(N^a\cap\mu\setminus\gamma(*)).
\end{equation}

Clearly $\beta(*)$ cannot be a successor ordinal, so $\beta(*)$ is either a singular limit ordinal or a regular cardinal.  We show that neither alternative is possible.
Our first step is to note that $\beta(*)\in N^b$. To see this, note that by first part of the proof of Claim~\ref{claim} there is a set
\begin{equation}
X\in\Sk(\{\mu_k,\mu\}\cup H)\cap [\mu]^{\leq\mu_k}
\end{equation}
with $\beta(*)\in X$.
Our construction guarantees that the set $X$ is in both $N^a$ and $N^b$. Since $X\subseteq N^a$,  it follows that $\beta^*$ can be defined in $N^b$ as $\min(X\setminus\gamma(*))$.

\begin{claim}
$\beta(*)$ is not a singular limit ordinal.
\end{claim}
\begin{proof}
Suppose not, and define
\begin{equation}
\kappa=\cf(\beta(*))<\beta(*).
\end{equation}
Note that $\kappa\in N^a\cap N^b$ as $\beta(*)$ is, and so by our choice of $\gamma(*)$ we must have
\begin{equation}
\kappa<\gamma(*)
\end{equation}
and
\begin{equation}
\label{cofinalneeded}
N^b\cap\kappa\subseteq N^a\cap\kappa.
\end{equation}
Fix $f\in N^a\cap N^b$ mapping $\kappa$ onto a cofinal subset of $\beta(*)$.  In $N^b$, we can find $\alpha<\kappa$ such that $\gamma(*)<f(\alpha)$.  But by (\ref{cofinalneeded}),
 we know $\alpha\in N^a$ as well and this yields a contradiction,  as
\begin{equation}
f(\alpha)\in N^a\cap [\gamma(*),\beta(*)).
\end{equation}
\end{proof}
\begin{claim}
$\beta(*)$ is not a regular cardinal.
\end{claim}
\begin{proof}
Suppose $\beta(*)$ were a regular cardinal.  Clearly $\mu^*\leq\beta(*)$ and so
\begin{equation}
\beta(*)\in N^a\cap N^b\cap [\mu_k^+,\mu)\cap\reg.
\end{equation}
Given our choice of $H_{k+1}$, the contruction guarantees that $N^a_{k+1}$ contains a function $f\in\prod(\mu\cap\reg)$ satisfying
\begin{equation}
\gamma(*)\leq\sup(N^b_k\cap\beta(*))<f(\beta(*))<\beta(*).
\end{equation}
But then of course we have $f(\beta(*))\in N^a\cap[\gamma(*),\beta(*))$ and we have a contradiction.
\end{proof}

The contradictions derived in the previous two claims establish Proposition~\ref{goalprop}, which then finishes our proof of Theorem~\ref{modelcover}.
\end{proof}
\end{proof}

\section{Conclusions}

In this section, we draw conclusions from Theorem~\ref{modelcover}.  We begin by deducing Shelah's $\cov$ vs. $\pp$ theorem, which pins down the relationship between $\ppgamma(\mu)$ and
the covering numbers discussed at the beginning of the previous section. This theorem as it appears originally in Theorem 5.4 in Chapter~II of {\em Cardinal Arithmetic} actually states something a little more general, but what we give here is the heart of the matter.

\begin{theorem}[The $\cov$ vs. $\pp$ Theorem]
\label{covpp}
Suppose $\sigma\leq\cf(\mu)<\theta<\mu$ are cardinals, with $\sigma$ uncountable and regular. Then
\begin{equation}
\cov(\mu,\mu,\theta,\sigma)=\pp_{\Gamma(\theta,\sigma)}(\mu).
\end{equation}
\end{theorem}
\begin{proof}
The fact that $\pp_{\Gamma(\theta,\sigma)}(\mu)\leq \cov(\mu,\mu,\theta,\sigma)$ is quite easy and done on page~88 of~\cite{cardarith} (and can also be derived from Corollary~\ref{triple} below).  In fact, that inequality does not require that $\sigma$ is uncountable at all.

For the other direction, let $\chi$ be a sufficiently large regular cardinal, and let $M$ be an elementary submodel of $H(\chi)$ of cardinality $\pp_{\Gamma(\theta,\sigma)}(\mu)$
 containing $\sigma$, $\theta$, and~$\mu$.  It suffices to show that $M$ satisfies the property $\circledast$.

So let $A\in M$ be a subset of $\mu\cap\reg$ with $|A|<\mu$.  We must show that $M\cap\prod A$ is $(\theta,\sigma)$-cofinal in $\prod A$ modulo the ideal of sets bounded in~$\mu$.  This is trivial if $A$ is bounded so we may assume $A$ is cofinal in $\mu\cap\reg$.

Let $\eta<\mu$ be as in Theorem~\ref{theorem1}. We may assume $\eta\in M$, and therefore so is $B=A\setminus\eta+1$.   By the conclusion of Theorem~\ref{theorem1}, if follows that
the $(\theta,\sigma)$-cofinality of $\prod B$ is bounded by $\pp_{\Gamma(\theta,\sigma)}(\mu)$.  The model $M$ will see a family $F\subseteq\prod B$ witnessing this, and since
\begin{equation}
\pp_{\Gamma(\theta,\sigma)}(\mu)+1\subseteq M,
\end{equation}
it follows that $M$ also contains every element of $F$ as well.  But this implies $M\cap\prod A$ is a $(\theta,\sigma)$-cofinal in $\prod A$ modulo the ideal of sets bounded in~$\mu$, and we have $\circledast$.
\end{proof}

A basic construction of Shelah gives us additional information, salvaging a little more of Theorem~3.5 of~\cite{400}.  The arguments used are due to Shelah.

\begin{corollary}
\label{triple}
With $\sigma$, $\mu$, and $\theta$ as in Theorem~\ref{covpp}, we have
\begin{equation}
\ppgamma(\mu)=\cf^\sigma_{<\theta}\left(\prod (\mu\cap\reg)/ J^{\bd}[\mu]\right)=\cov(\mu,\mu,\theta,\sigma).
\end{equation}
\end{corollary}
\begin{proof}
The fact that the third cardinal is less than or equal to the first is the previous theorem, and that is the only place where we need $\sigma$ to be uncountable.
The other two inequalities necessary for the result hold without this assumption via some standard arguments presented below for completeness.

If $F$ is $(\theta,\sigma)$ cofinal in $\prod(\mu\cap\reg)$ modulo the bounded ideal, then $|F|$ must bound all cardinals in $\PP_{\Gamma(\theta,\sigma)}(\mu)$, and hence
\begin{equation}
\ppgamma(\mu)\leq |F|.
\end{equation}

Suppose now that  $\mathcal{P}\subseteq[\mu]^{<\mu}$ is a $\sigma$-cover of $[\mu]^{<\theta}$. Given $B\in\mathcal{P}$, define a function $f_B$ in $\prod(\mu\cap\reg)$ by
\begin{equation}
f_B(\tau)=
\begin{cases}
\sup(B\cap\tau)+1  &\text{if $|A|<\tau$}\\

0  &\text{otherwise.}
\end{cases}
\end{equation}
We claim the collection $\{f_B:B\in\mathcal{B}\}$ is $(\theta,\sigma)$-cofinal in $\prod(\mu\cap\reg)$ modulo the bounded ideal.

Given a function $f\in\prod(\mu\cap\reg)$ and a cofinal $A\subseteq\mu\cap\reg$ with $|A|<\theta$, we can find $\sigma^*<\sigma$ and sets $B_\alpha\in \mathcal{P}$ for each $\alpha<\sigma^*$ such that
\begin{equation}
\{f(\tau):\tau\in A\}\subseteq\bigcup_{\alpha<\sigma^*}B_\alpha.
\end{equation}
Since $\alpha^*<\sigma\leq\cf(\mu)$ and each $B_\alpha$ is of cardinality less than~$\mu$, we can find a cardinal~$\zeta<\mu$ such that
\begin{equation}
\alpha<\alpha^*\Longrightarrow|B_\alpha|<\zeta.
\end{equation}
For $\tau>\zeta$ in $\mu\cap\reg$, we can choose $\alpha$ with $f(\tau)\in B_\alpha$ and then
\begin{equation}
f(\tau)\leq\sup(B_\alpha\cap\tau)<f_{B_\alpha}(\tau),
\end{equation}
and required.
\end{proof}

We note that the preceding corollary hides a nice quantifier exchange: there is a {\em single} family of $\ppgamma(\mu)$ functions in $\prod\mu\cap\reg$ whose restrictions are cofinal in $\prod A/J$ whenever $A$ is cofinal in~$\mu$ with $|A|<\theta$ and $J$ is an ideal on $A$ extending the bounded ideal.

As another easy corollary, we note the following result of Shelah to which we alluded in the introduction.

\begin{corollary}
\label{6something}
Suppose $\sigma$ is an uncountable regular cardinal, and $\mu$ is singular with $\sigma\leq\cf(\mu)$. Then
\begin{equation}
\pp_{\Gamma(\mu,\sigma)}(\mu)=\cov(\mu,\mu,\mu,\sigma).
\end{equation}
\end{corollary}
\begin{proof}
If $\langle\theta_\alpha:\alpha<\cf(\mu)\rangle$ is an increasing sequence of regular cardinals cofinal in~$\mu$, each greater than~$\cf(\mu)$, then
\begin{align*}
\cov(\mu,\mu,\mu,\sigma) &=\sup\{\cov(\mu,\mu,\theta_\alpha,\sigma):\alpha<\cf(\mu)\}\\
&=\sup\{\pp_{\Gamma(\theta_\alpha,\sigma)}(\mu):\alpha<\cf(\mu)\}\\
&=\pp_{\Gamma(\mu,\sigma)}(\mu).
\end{align*}
\end{proof}

We now come to a key result for our investigation, a theorem that unlocks several consequences of failures of compactness. 

\begin{theorem}
\label{bad}
Suppose $\aleph_0<\sigma\leq\cf(\mu)<\theta<\mu$ with $\sigma$ and $\theta$ regular, and $\ppgamma(\mu)$ is a weakly inaccessible cardinal $\kappa$.
If $\kappa\notin\PP_{\Gamma(\theta,\sigma)}(\mu)$, then there is a progressive $A\subseteq\mu\cap\reg$ such that
 $\pcf_{\Gamma^*(\theta,\sigma)}(A)$ is an unbounded subset of~$\kappa$
\end{theorem}
\begin{proof}
Suppose this fails, and let $\chi$ be a sufficiently large regular cardinal, and let $M$ be an elementary submodel of $H(\chi)$ such that
\begin{itemize}
\item $\mu+1\subseteq M$,
\sk
\item $|M|<\kappa$, and
\sk
\item $M\cap\kappa$ is an initial segment of $\kappa$.
\sk
\end{itemize}
This can be arranged easily as $\kappa$ is weakly inaccessible, and we now work to show that $M$ satisfies the condition $\circledast$ of Theorem~\ref{modelcover}.

If $A$ a cofinal subset of $\mu\cap\reg$ with $|A|<\mu$, then $\pcf_{\Gamma^*(\theta,\sigma)}(A)$ must be a bounded subset $\kappa$ -- it cannot be unbounded in~$\kappa$ because we have assumed the theorem fails, and since $\pcf_{\Gamma^*(\theta,\sigma)}(A)\subseteq\PP_{\Gamma(\theta,\sigma)}(\mu)$ it cannot contain anything of cardinality $\kappa$ or larger.  This means that
\begin{equation}
\cf^\sigma_{<\theta}\left(\prod A/ J^{\bd}[\mu]\right)<\kappa
\end{equation}
for any such $A$.

Since $M\cap\kappa$ is an initial segment of $\kappa$, given any such $A\in M$ we know the model contains every member of some family $(\theta,\sigma)$-cofinal in $\prod A/J^{\bd}[\mu]$, and $\circledast$ follows immediately.  But this is absurd, as Theorem~\ref{modelcover} tells us
\begin{equation}
\cov(\mu,\mu,\theta,\sigma)\leq |M|<\kappa=\ppgamma(\mu),
\end{equation}
contradicting Theorem~\ref{covpp}.
\end{proof}

Given Theorem~\ref{bad}, we can start deducing consequences of failures of compactness.  The first thing to note is that such a set $A$ contradicts a conjecture of Shelah:  Conjecture~1.10 of~\cite{666} asserts that a progressive set of regular cardinals can never have an inaccessible accumulation point.  This conjecture is a fundamental one for pcf theory: if it holds, then $\cf(\prod\pcf A)=\cf(\prod A)$ for every progressive~$A$, and if it fails  then one can force a counterexample to this. Shelah discusses this on page 5 of~\cite{666}\footnote{See also the last section of~\cite{420}.}, and notes there that this conjecture is a significant dividing line between chaos and order.  For our purposes,  we just observe that this means getting a counterexample to compactness for~$\PP_{\Gamma(\theta,\sigma)}(\mu)$ will  be quite hard. As he says in the introduction to Section 3 of~\cite{371}, ``Whether $\pcf(\mathfrak{a})$ may have an inaccessible accumulation point remains a mystery for me.''.

With a little more work, we can show that the assumptions of Theorem~\ref{bad} have implications for cardinal arithmetic below~$\mu$ as well. This requires the following lemma, which captures a compactness-type property of $\pcfsig(A)$.

\begin{lemma}
\label{curious}
Suppose $A$ is a progressive set of regular cardinals and $\sigma$ is regular.  Then either $\pcfsig(A)$ has a maximum element, or $\sup\pcfsig(A)$ is singular of cofinality less than $\sigma$.
\end{lemma}
\begin{proof}
By Claim 6.7F of~\cite{430}, there is a set $B\subseteq\pcfsig(A)$ with
\begin{equation}
|B|<\sigma,	
\end{equation}
and
\begin{equation}
A\subseteq \bigcup\{B_\lambda[A]:\lambda\in B\}.	\end{equation}
We claim
\begin{equation}
\sup\pcfsig(A)=\sup(B).	
\end{equation}
Since $B\subseteq\pcfsig(A)$, it suffices to show that the $<\sigma$-cofinality of $\prod A$ is at most $\sup(B)$.  For each $\tau\in B$, we know that the cofinality of $\prod B_\tau$ is $\tau$, and therefore we can fix $F_\tau\subseteq\prod A$ of cardinality $\tau$ such that the restrictions $f\restr B_\tau[A]$ for $f\in F_\tau$ are cofinal in $\prod B_\tau$. Then
\begin{equation}
F:=\bigcup\{F_\tau:\tau\in B\}	\end{equation}
has cardinality $\sup(B)$ and is clearly $<\sigma$-cofinal in~$\prod A$.  If $\cf(\sup(B))\geq\sigma$, then $B$ must have a maximum element as $|B|<\sigma$ and the result follows.
\end{proof}

\begin{proposition}
Suppose $A$ is a progressive set of regular cardinals and $\sup\pcfgamma(A)$ is a weakly inaccessible inaccessible cardinal~$\kappa$. Then
\begin{equation}
\cf\left([A]^{<\theta},\subseteq\right)<\kappa\Longrightarrow \kappa\in\pcfgamma(A).
\end{equation}
\end{proposition}
\begin{proof}
For each $\tau\in\pcfgamma(A)$ we can choose a set $D_\tau\in [A]^{<\theta}$ such that
\begin{equation}
\tau\in\pcfsig(D_\tau).
\end{equation}
If $\cf\left([A]^{<\theta},\subseteq\right)<\kappa$ then there is a $D\in[A]^{<\theta}$ for which
\begin{equation}
\kappa=\sup\pcfsig(D),
\end{equation}
as we can choose $D$ such that $\{\tau\in\pcfgamma(A):D_\tau\subseteq D\}$ has cardinality $\kappa$.  But then
\begin{equation}
\kappa\in\pcfsig(D)
\end{equation}
by way of Lemma~\ref{curious}, and since $|D|\in[A]^{<\theta}$ it follows that
\begin{equation}
\kappa\in\pcfgamma(A),
\end{equation}
as required
\end{proof}

Looking at the preceding proposition in the context of Theorem~\ref{bad}, we obtain the follow corollary which speaks on cardinal arithmetic below~$\mu$.

\begin{corollary}
Suppose $\aleph_0<\sigma\leq\cf(\mu)<\theta<\mu$ with $\sigma$ and $\theta$ regular, and $\ppgamma(\mu)$ is a weakly inaccessible cardinal $\kappa$.
If $\kappa\notin\PP_{\Gamma(\theta,\sigma)}(\mu)$, then there is a cardinal $\chi<\mu$ such that
\begin{equation}
\kappa\leq \cf\left([\chi]^{<\theta},\subseteq\right).
\end{equation}
In particular, $\mu$ cannot be a strong limit cardinal.
\end{corollary}

Finally, we turn to the compactness result promised in the introduction, rephrased in more standard terminology.\footnote{To get a positive answer to Question~1, just take $\sigma=\cf(\mu)$ in Theorem~\ref{main}.}

\begin{theorem}
\label{main}
Suppose $\sigma$ is an uncountable regular cardinal, $\mu$ is singular with $\cf(\mu)\geq\sigma$, and $\pp_{\Gamma(\mu,\sigma)}(\mu)$ is a weakly inaccessible cardinal~$\kappa$.  Then $\kappa\in\PP_{\Gamma(\mu,\sigma)}(\mu)$.
\end{theorem}
\begin{proof}
Let $\langle\theta_\alpha:\alpha<\cf(\mu)\rangle$ be an increasing sequence of regular cardinals cofinal in~$\mu$.  Since
\begin{equation}
\pp_{\Gamma(\mu,\sigma)}(\mu)=\sup\{\pp_{\Gamma(\theta_\alpha,\sigma)}(\mu):\alpha<\cf(\mu)\},
\end{equation}
there must be a regular $\theta<\mu$ for which
\begin{equation}
\kappa=\ppgamma(\mu).
\end{equation}
If $\kappa\in\PP_{\Gamma(\theta,\sigma)}(\mu)$ we are done, so assume this does not happen. This is exactly the hypothesis of Theorem~\ref{bad}, so we may fix a set~$A$ as in the conclusion, that is, with
\begin{equation}
\kappa=\sup\pcf_{\Gamma^*(\theta,\sigma)}(A).
\end{equation}
We will finish by showing
\begin{equation}
\label{finish}
\kappa\in\PP_{\Gamma(|A|^+,\sigma)}(\mu).
\end{equation}

Proposition~\ref{tangerine} tells us
\begin{equation}
\kappa=\lim_{\zeta<\mu}\sup\pcfgamma(A\setminus\zeta)
\end{equation}
and so without loss of generality,
\begin{equation}
\zeta<\mu\Longrightarrow\kappa=\sup\pcfgamma(A\setminus\zeta),
\end{equation}
and in particular,
\begin{equation}
\zeta<\mu\Longrightarrow\kappa\leq\sup\pcfsig(A\setminus\zeta).
\end{equation}

By Lemma~\ref{curious}, it follows that for each $\zeta<\mu$, the set $\pcfsig(A\setminus\zeta)$ must contain a cardinal greater than or equal to~$\kappa$.
Thus, we can define
\begin{equation}
\lambda_\zeta:=\min\left(\pcfsig(A\setminus\zeta)\setminus\kappa\right).
\end{equation}
The sequence $\langle \lambda_\zeta:\zeta<\mu\rangle$ is non-increasing hence eventually constant, and by removing an initial segment of $A$ if needed, we may assume the sequence is constant, say with value $\lambda$.  But this means
\begin{equation}
\lambda\in\pcf_{\Gamma(|A|^+,\sigma)}(A\setminus\zeta)\text{ for all }\zeta<\mu,
\end{equation}
and so by Corollary~\ref{starchar}
\begin{equation}
\lambda\in\pcf_{\Gamma^*(|A|^+,\sigma)}(A)\subseteq\PP_{\Gamma(|A|^+,\sigma)}(\mu).
\end{equation}
Since $\lambda\geq\kappa$ and $\PP_{\Gamma(|A|^+,\sigma)}(\mu)$ is an interval of regular cardinals, we have~(\ref{finish}).
\end{proof}

We make a couple of notes to finish this section.  First, is that the combination of Corollary~\ref{6something} and Theorem~\ref{main} gives us the following:

\begin{corollary}
Suppose $\sigma$ is an uncountable regular cardinal, and~$\mu$ is singular of cofinality at least~$\sigma$. Then
\begin{equation}
\pp_{\Gamma(\mu,\sigma)}(\mu)=^+\cov(\mu,\mu,\mu,\sigma).
\end{equation}
\end{corollary}

The two cardinals are equal by Corollary~\ref{6something}, and the notation ``$=^+$'' is used by Shelah to mean that the supremum on the left side is attained if it is regular. This is just a restatement of Theorem~\ref{main} using other notation, but it emphasizes the connection between what we do here and Shelah's work in~\cite{400}.

The second note we make is simply that if something similar to Lemma~\ref{curious} were to hold for $\pcfgamma(A)$, then we could push through compactness for $\Gamma(\theta,\sigma)$ at~$\mu$ by the same proof.  This does not seem to be very likely, although forcing a counterexample is probably very difficult.  Certainly obtaining a set $A$ as in the conclusion of Theorem~\ref{bad} is beyond our current technology.

\section{Questions}

One of the difficulties with pcf theory is that most of the open problems are really difficult, and that will almost certainly be true for many of the questions we pose here, but perhaps not all of them.

\begin{question}
Suppose $\sigma<\theta$ are regular cardinals and~$\mu$ is singular with
\begin{equation}
\aleph_0<\sigma\leq\cf(\mu)<\theta<\mu.
\end{equation}
If $\ppgamma(\mu)$ is a weakly inaccessible cardinal~$\kappa$, must $\kappa$ be in $\PP_{\Gamma(\theta,\sigma)}(\mu)$?  In other words, is
\begin{equation}
\ppgamma(\mu)=^+\sup\PP_{\Gamma(\theta,sigma)}(\mu)?
\end{equation}
\end{question}

One can also ask about the role played by $\sigma$ in our results, in particular, whether the restriction to uncountable $\sigma$ is necessary. For example, the following is natural (and asked by Shelah in Section~1 of~\cite{355}) and makes sense even for~$\mu$ of countable cofinality:
\begin{question}
Is $\pp(\mu)=^+\sup\PP(\mu)$ for every singular cardinal~$\mu$?
\end{question}
\noindent (Again, this is of interest only in the situation where $\pp(\mu)$ is a weakly inaccessible cardinal.)

The role of uncountable cofinality in the cov vs. pp Theorem is also a topic of interest, and Shelah spends much of~\cite{400} trying to eliminate the assumption.  The most concise question to be asked here is:

\begin{question}
\label{other}
If $\mu$ is singular of countable cofinality, is $\pp(\mu)=\cov(\mu,\mu, \aleph_1, 2)$?
\end{question}
There is much more discussion of these matters in the {\em Analytical Guide} [E:12] appendix of~\cite{cardarith}, and also in the first section of~\cite{666}.  Shelah has several partial results in the various continuations of the book.

We can ask variants of the previous question that may turn out to be easier to resolve.  For example, the proof of Corollary~\ref{triple} shows us 
\begin{equation}
\pp(\mu)\leq\cf^{\aleph_0}_{<\aleph_1}\left(\prod\mu\cap\reg\right)\leq\cov(\mu,\mu,\aleph_1,\aleph_0)
\end{equation}
for $\mu$ singular of countable cofinality.  Clearly the middle of this is just the ($<\aleph_1$)-cofinality of $\prod \mu\cap\reg$, while the covering number is equal to $\cov(\mu,\mu,\aleph_1, 2)$, the same that appears in Question~\ref{other}.  Perhaps it is easier to show that two of these three cardinals are equal to each other?

Finally, there is Shelah's Conjecture 1.10 which remains untouched:

\begin{question}
Suppose $A$ is a progressive set of regular cardinals. Can $\pcf(A)$ have a weakly inaccessible point of accumulation?  Equivalently, does ZFC prove
\begin{equation}
\cf\left(\prod A\right)=\cf\left(\prod \pcf(A)\right)
\end{equation}
for every progressive set $A$ of regular cardinals?
\end{question}

\end{document}